\documentclass[11pt]{amsart}
\newtheorem{thm}{Theorem}[section]
\newtheorem{lem}[thm]{Lemma}
\newtheorem{remark}[thm]{Remark}
\newtheorem{cor}[thm]{Corollary}

\theoremstyle{plain}

\theoremstyle{definition}

\usepackage{amssymb,latexsym,amscd,amsmath,amsthm}
\usepackage[all]{xy}
\xyoption{all}
\usepackage{epsfig}
\usepackage{graphicx}
\usepackage{color}
\usepackage{geometry}
\usepackage{multirow}

\usepackage[hang,small,bf]{caption}
\usepackage{palatino}

\newcommand{\R}{\mathbf{R}}
\newcommand{\Z}{{\mathbf{Z}}}

\newcommand{\C}{{\mathbf{C}}}
\newcommand{\cp}{\mathbf{CP}}
\newcommand{\rp}{\mathbf{RP}}
\newcommand{\hp}{\mathbf{HP}}
\newcommand{\sph}{\mathbf{S}}

\newcommand{\G}{\operatorname{G}}
\newcommand{\SU}{\operatorname{SU}}

\renewcommand{\H}{\mathbf{H}}
\newcommand{\syp}{\operatorname{Sp}}
\newcommand{\SO}{\operatorname{SO}}
\newcommand{\Spin}{\operatorname{Spin}}

\newcommand{\U}{\operatorname{U}}

\newcommand{\Aut}{\operatorname{Aut}}

\newcommand{\tr}{\operatorname{tr}}

\renewcommand{\lim}[1]{\mathop{\underset{#1} {\underset \longleftarrow
{\text{\rm lim}}}}}

\newcommand{\comp}{{\text{ $\scriptstyle \circ$ }}}
\newcommand{\Ad}{\operatorname{Ad}}

\newcommand{\T}{\operatorname{T}}

\newcommand{\diag}{\operatorname{diag}}

\newcommand{\lra}{\longrightarrow}
\newcommand{\mf}[1]{\mathfrak{#1}}

\def\x{\times}
\def\<{\langle}
\def\>{\rangle}
\def\met{\<\, ,\, \>}

\newcommand{\MMet}[2]{\< #1 , #2 \>}

\renewcommand{\geq}{\geqslant}
\renewcommand{\leq}{\leqslant}

\numberwithin{equation}{section}

\begin{document}

\newcommand{\spacing}[1]{\renewcommand{\baselinestretch}{#1}\large\normalsize}
\spacing{1.14}

\title[Riemannian submersions from simple, compact Lie groups]{Riemannian submersions from simple,\\ compact Lie groups}

\author[M.\ Kerin and K.\ Shankar]{Martin Kerin$^\dagger$ and Krishnan Shankar$^\ast$}

\thanks{$^\dagger$ Supported by SFB 478 at the University of M\"unster.}
\thanks{$^\ast$ Partially supported by an NSF grant and SFB 478 at the University of M\"unster.}

\address{Mathematisches Institut, Einsteinstr. 62, 48149 M\"unster, Germany}
\email{m.kerin@math.uni-muenster.de}

\address{Department of Mathematics\\ University of Oklahoma\\ Norman
   \\ OK 73019}
\email{shankar@math.ou.edu}


\begin{abstract}
In this paper we construct infinitely many examples of a Riemannian submersion from a simple, compact Lie group $G$ with bi-invariant metric onto a smooth manifold that cannot be a quotient of $G$ by a group action. This partially addresses a question of K.\ Grove's about Riemannian submersions from Lie groups.
\end{abstract}

\maketitle
\vspace{-.3cm}
\begin{center}
\textit{Dedicated to J.-H.\ Eschenburg on the occasion of his sixtieth birthday}
\end{center}

\normalsize
\thispagestyle{empty}

\section*{Introduction}

Riemannian submersions (which we always assume to have connected fibers) are fundamentally important in several areas of Riemannian geometry.  For instance, it is a classical and important problem in Riemannian geometry to construct Riemannian manifolds with positive or non-negative sectional curvature.  While there are a few methods, the most abundant source of examples comes via submersions from compact Lie groups (see \cite{Zi} for a survey).  In addition, many of the known examples of Einstein manifolds are constructed via Riemannian submersions (see \cite{Be}). Moreover, in order to prove the Diameter Rigidity Theorem for positively curved manifolds (\cite{GG1}, \cite{Wi}), a classification of Riemannian submersions from spheres equipped with a round metric was required (\cite{GG2}, \cite{Wi}).  As it turns out, all that can arise are Hopf fibrations.  In the special case where the fibers are totally geodesic, this classification had been achieved in \cite{Es1} (see also \cite{Ra}).  It is natural, therefore, to ask for a classification of Riemannian submersions from \emph{special} Riemannian manifolds.  In \cite{Es2} the author classified Riemannian submersions with totally geodesic fibers from $\cp^n$ equipped with the usual Fubini-Study metric.  Riemannian submersions from (flat) Euclidean space $\R^{n+k}$ were classified in \cite{GW}, where it was shown that the base must be diffeomorphic to $\R^n$ and the quotient of $\R^{n+k}$ by an isometric $\R^k$ action. Given the many geometric situations in which Riemannian submersions from Lie groups arise, it is therefore natural that one should address the following problem (\cite[Problem 5.4]{gro}): Determine the structure of all Riemannian submersions from $G$, where $G$ is a compact Lie group with a bi-invariant metric.
\smallskip

Until recently the only known Riemannian submersions from compact Lie groups equipped with bi-invariant metrics arose as biquotient submersions, namely Riemannian submersions from $(G, \met_0)$ given by the quotient of $G$ by a two-sided, free, isometric action of some closed subgroup of $G \x G$.  This changed with the discovery of a single example of a Riemannian submersion, $\SO(16) \rightarrow \sph^8$ (see Section 1), such that the base is not a quotient of the total space $G$ by a free group action.  In this paper we construct infinitely many examples of Riemannian submersions, $G \rightarrow B$, where $G$ is a simple compact Lie group equipped with a bi-invariant metric, and such that $B$ is not a quotient $G/U$ by any subgroup $U \subseteq {\rm Diff}(G)$; see Table \ref{list}.

\begin{table}[!ht]
\begin{tabular}{|lcl|c|} \hline
\multicolumn{1}{|c}{$\boldsymbol{G}$} & \multicolumn{1}{c}{$\boldsymbol{\lra}$} & \multicolumn{1}{l|}{$\boldsymbol{B}$} & \multicolumn{1}{|c|}{$\boldsymbol{n}$}\\ \hline \hline
$\SO(16)$ & $\lra$ & $\sph^8$ & - \\ \hline
$\SO(2n)$ & $\lra$ & $\sph^{2n-2}$ & $n \geq 4$\\ \hline
$\SU(2n)$ & $\lra$ & $\sph^{4n-3}$ & $n \geq 3$\\ \hline
$\SU(2n)$ & $\lra$ & $\cp^{2n-2}$ & $n \geq 3$\\ \hline
$\SO(4n)$ & $\lra$ & $V_{3}(\R^{4n-1})$ & $n \geq 3$ \\ \hline
$\SO(4n)$ & $\lra$ & $\sph^1 \backslash V_3(\R^{4n-1})$ & $n \geq 3$ \\ \hline
\end{tabular}
\vspace{.1cm}
\caption{Riemannian submersions $G \lra B$, where $B$ is not a quotient of $G$.}
\label{list}
\end{table}

It is a pleasure to thank Christoph B\"ohm, Luigi Verdiani, Burkhard Wilking and Wolfgang Ziller for many helpful discussions. We would also like to thank AIM and the organizers of the workshop on non-negative curvature in September 2007 that was, at least in part, responsible for our interest in this question. We would also like to acknowledge our debt to the paper \cite{Ke} which was invaluable as a reference for all our computations. Finally, we would both like to thank the Mathematics Institute and SFB 478 at the University of M\"unster for their hospitality and support.

\section{A Riemannian submersion, $\SO(16) \rightarrow \sph^8$}
\label{hopf}

In 2007, at a week long workshop at AIM in Palo Alto, one of the working groups constructed this example. Since we were both at the workshop, we would like to acknowledge the contribution of the members of the workshop in stimulating interest in the problem, especially Corey Hoelscher, Marius Munteanu, Craig Sutton, Kris Tapp and Wolfgang Ziller.

Consider the Hopf fibration
$$
\sph^7 \rightarrow \sph^{15} \rightarrow \sph^8.
$$
It is well-known that the round metric on $\sph^{15}$ induces a Riemannian submersion onto $\sph^8$.  Moreover, the isometry group of $(\sph^{15}, g_{\rm round})$ is $\SO(16)$, which also acts transitively on $\sph^{15}$ with isotropy subgroup $\SO(15)$.  Since $\SO(16)/\SO(15)$ is isotropy irreducible, it follows that the bi-invariant metric $\met_0$ on $\SO(16)$ induces the round metric on $\sph^{15}$ via a Riemannian submersion.  We may therefore compose these two Riemannian submersions to yield a Riemannian submersion $(\SO(16), \met_0) \lra \sph^{8}$.

However, this submersion is not the result of a free action by some Lie group $U$.  In particular, this is not a biquotient submersion.  If there were such a $U$, then $\dim(U) = \dim(\SO(16)) - 8 = 112$. From the long exact sequence of homotopy groups associated to the fibration $U \lra \SO(16) \lra \sph^8$, we see that $\pi_3(U) = \pi_3(\SO(16)) = \Z$ and $\pi_1(U) = \pi_1(\SO(16)) = \Z_2$. Since $\pi_3$ for a Lie group is the number of simple factors from which we conclude that $U$ is a simple, compact Lie group of dimension 112. A quick look at the classification of simple Lie groups reveals that there is no such group.


\section{The basic construction}
\label{basic}

The above example was, to date, the only known Riemannian submersion from a compact Lie group with bi-invariant metric that is not the result of a group action.  It is natural to wonder if this example is special in some way.  For instance, perhaps the construction relies on the fact that the Hopf fibration, $\sph^{15} \rightarrow \sph^8$, is not a principal bundle.  It turns out that this is not the case.  The important observation one should make is that the Hopf fibration is, in fact, a homogeneous fibration coming from the triple $\Spin(7) \subseteq \Spin(8) \subseteq \Spin(9)$:
$$
\Spin(8)/\Spin(7)=\sph^7 \hookrightarrow \Spin(9)/\Spin(7)=\sph^{15} \rightarrow \Spin(9)/\Spin(8) = \sph^8.
$$
In particular, $\sph^{15}$ may be written as a homogeneous space in two different ways.

There is another subtlety of which one should be wary.  The round metric on $\sph^{15}$ is not isometric to the normal homogeneous metric on the quotient $\Spin(9)/\Spin(7)$.  Thus, in order to combine the submersions so that the composed map is a Riemannian submersion, one has to choose the homogeneous metric on $\Spin(9)/\Spin(7)$ carefully.  The isotropy representation on $\Spin(9)/\Spin(7)$ has two irreducible summands (of dimensions $7$ and $8$).  Hence, as we shall see in Section \ref{hommet}, there is a two parameter family of homogeneous metrics on $\Spin(9)/\Spin(7)$.  It is possible to choose these parameters so that $\Spin(9)/\Spin(7)$ is equipped with the round metric.  Now one of the irreducible summands is tangent to the base $\Spin(9)/\Spin(8)$ and also irreducible under the $\Spin(8)$ isotropy action.  Therefore the restriction of the metric on $\Spin(9)/\Spin(7)$ to this isotropy summand yields a Riemannian submersion onto $\Spin(9)/\Spin(8)$.

The construction of all examples in this paper relies on putting together the two key ideas indicated above, namely:

First we look for homogeneous spaces that can be represented as the quotient of two distinct (simple) groups. Given a homogeneous space that can be represented as $G/K_1 = K_2/H$, we then proceed to find intermediate subgroups $H \subseteq L \subseteq K_2$ which give rise to a homogeneous fibration,
$$
L/H \lra K_2/H \lra K_2/L
$$
We now have two submersions, $\pi_1: G \lra G/K_1$ and $\pi_2: K_2/H \lra K_2/L$ which we compose to obtain a submersion $\pi: G \lra K_2/L$. Then we show, in some cases, that there is no $U \subseteq {\rm Diff}(G)$ such that $G/U = K_2/L$.

The second idea is to find a homogeneous metric on $K_2/H$ that is isometric to the normal homogeneous metric on $G/K_1$ and which induces a well-defined homogeneous metric on $K_2/L$ so that the map $\pi_2: K_2/H \lra K_2/L$ is a Riemannian submersion.  It is not always possible to do this (see for instance Section \ref{evensphere2}).  Whenever we can find such a metric, the submersion $\pi = \pi_2 \circ \pi_1 :G \lra K_2/L$ is Riemannian.
\smallskip

Let us examine the first part of the construction suggested above.  Suppose $\mf{g}$ is a compact Lie algebra with sub-algebras $\mf{k}_1, \mf{k}_2$ such that $\mf{g} = \mf{k}_1 + \mf{k}_2$. Then this is equivalent to the following: let $G$ be the simply connected, compact, Lie group with Lie algebra $\mf{g}$ and let $K_1, K_2$ be the closed subgroups in $G$ corresponding to the sub-algebras $\mf{k}_1$ and $\mf{k}_2$ respectively. Then we have the homogeneous space identities, $G/K_1 = K_2/(K_1 \cap K_2)$ and $G/K_2 = K_1/(K_1 \cap K_2)$. In 1962, A.\ L.\ Oni\v s\v cik classified all $(\mf{g}, \mf{k}_1, \mf{k}_2)$, where $\mf{g}$ is a simple, compact Lie algebra \cite{On}; all his spaces are given in Table \ref{onishchik} in the Appendix.

Now suppose there is a subgroup $U\subseteq {\rm Diff}(G)$ such that the base $K_2/L$ may be realized as the quotient $G/U$. Then from the long exact homotopy sequence of the fibration, $U \rightarrow G \rightarrow K_2/L$, we may compute the homotopy groups of $U$. Moreover, since we know $G$ and $K_2/L$, we also know the dimension of $U$. From this we can determine the (local) decomposition of $U$ into simple and torus groups. Every compact, connected Lie group $U$ is finitely covered by a Lie group diffeomorphic to $\T^k \x \widetilde U$, where $\T^k$ is a torus and $\widetilde U$ is a product of compact, connected, simply connected, simple Lie groups.  Now, since $\pi_1(\T^k \x \widetilde U) = \Z^k$ injects into $\pi_1(U)$ under the homomorphism induced by the covering, it follows that if we can determine $\pi_1(U)$ then we will know the rank $k$ of the torus $\T^k$. In addition, if we can find $\pi_3(U)$ then we will have determined the number of simple factors in $\widetilde U$.  

If we assume that the simple factors of $\widetilde U$ have dimension large enough, then determining $\pi_5(U) = \pi_5(\widetilde U)$ will allow us to decide which classical Lie groups are possible for the simple factors.  This is achieved via 
the isomorphisms $\pi_5(\Spin(n)) \cong \pi_5(\mathbf{O}) \cong 0$ if $n \geq 7$; $\pi_5(\SU(n)) \cong \pi_5(\mathbf{U}) \cong \Z$ if $n \geq 3$; and $\pi_5(\syp(n)) \cong \pi_5(\mathbf{Sp}) \cong \Z_2$ if $n\geq 1$,
where $\textbf{O}, \textbf{Sp}, \textbf{U}$ denote the stable (infinite dimensional) limits of the corresponding Lie groups (see \cite[pgs.\ 466--467]{Bre} for more details).  The remaining possibilities for simple factors are low-dimensional classical Lie groups and the exceptional Lie groups.


By examining the various examples in Oni\v s\v cik's list, we see that in some of the cases such a $U$ is not possible for dimension reasons (since all possible finite covers $T^k \times \widetilde U$ of $U$ with $\pi_1(\widetilde U) = 0$ may be determined as above).   This yields candidate (topological) submersions which need to be examined metrically.  Some of these candidates are listed in Table \ref{table1}.  Evidently, the example in Section \ref{hopf} falls neatly into this scheme.

\begin{table}[!ht]
\small
\begin{tabular}{|l|l|l|l|c|} \hline
\multicolumn{1}{|c|}{$\boldsymbol{G}$} & \multicolumn{1}{|c|}{$\boldsymbol{G/K_1}$} & \multicolumn{1}{|c|}{$\boldsymbol{K_2/H}$} & \multicolumn{1}{|c|}{$\boldsymbol{L}$} & $\boldsymbol{K_2/L}$\\ \hline \hline
\multicolumn{5}{|c|}{$\boldsymbol{G/K_1}$ \textbf{symmetric}} \\ \hline \hline
$\SO(16)$ & $\SO(16)/\SO(15)$ & $\Spin(9)/\Spin(7)$ & $\Spin(8)$ & $\sph^8$\\ \hline
$\underset{(n\geq 4)}{\SO(2n)}$ & $\SO(2n)/\U(n)$ & $\SO(2n-1)/\U(n-1)$ & $\SO(2n-2)$ & $\sph^{2n-2}$ \\ \hline
$\underset{(n\geq 3)}{\SU(2n)}$ & $\SU(2n)/\syp(n)$ & $\SU(2n-1)/\syp(n-1)$ & $\SU(2n-2)$ & $\sph^{4n-3}$ \\ \hline
$\underset{(n\geq 3)}{\SU(2n)}$ & $\SU(2n)/\syp(n)$ & $\SU(2n-1)/\syp(n-1)$ & $\U(2n-2)$ & $\cp^{2n-2}$ \\ \hline
\multicolumn{5}{c}{} \\ \hline
\multicolumn{5}{|c|}{$\boldsymbol{G/K_1}$ \textbf{non-symmetric}} \\ \hline \hline
$\underset{(n\geq 4)}{\SO(2n)}$ & $\SO(2n)/\SU(n)$ & $\SO(2n-1)/\SU(n-1)$ & $\SO(2n-2)$ & $\sph^{2n-2}$ \\ \hline
$\underset{(n\geq 3)}{\SO(4n)}$ & $\SO(4n)/\syp(n)\syp(1)$ & $\SO(4n-1)/\syp(n-1)\syp(1)$ & $\SO(4n-2)$ & $\sph^{4n-2}$ \\ \hline
$\underset{(n\geq 3)}{\SO(4n)}$ & $\SO(4n)/\syp(n)\syp(1)$ & $\SO(4n-1)/\syp(n-1)\syp(1)$ & $\SO(4n-3)$ & $T^1\sph^{4n-2}$ \\ \hline
$\underset{(n\geq 3)}{\SO(4n)}$ & $\SO(4n)/\syp(n)\syp(1)$ & $\SO(4n-1)/\syp(n-1)\syp(1)$ & $\SO(4n-4)$ & $V_{3}(\R^{4n-1})$ \\ \hline
\end{tabular}
\vspace{.1cm}
\caption{Candidate submersions $G \lra K_2/L$ that are not group quotients.}
\label{table1}
\end{table}
\normalsize

Besides the candidates listed above we also have $\SO(4n)/\syp(n)\U(1) = \SO(4n-1)/\syp(n-1)\U(1)$ and $\SO(4n)/\syp(n) = \SO(4n-1)/\syp(n-1)$. Each of these yields the same base spaces as the last three examples in Table \ref{table1}. The bases $B_1 = T^1\sph^{4n-2}$ and $B_2 = V_{3}(\R^{4n-1})$ admit a free diagonal $\SO(2)$ action from the left which is isometric for any homogeneous metric on $B_1$, $B_2$ respectively.  Thus we have a (topological) submersion $\SO(4n) \lra \SO(2) \backslash B_i$, $i= 1,2$, which will be Riemannian if $\SO(4n) \lra B_i$, $i=1,2$ respectively, is Riemannian.  However, as we shall see, not all of the candidates in Table \ref{table1} yield Riemannian submersions from $G$ onto the base.  In order to complete the picture we also need to address the metric part of the construction.

\section{Homogeneous metrics on $G/H$}
\label{hommet}

Given a compact, semisimple Lie group $G$ and a closed subgroup $H\subseteq G$ one has a natural decomposition of the Lie algebra $\mf{g}$ into invariant subspaces under the adjoint action of $H$: $\mf{g} = \mf{h} \oplus \mf{m}$, where $\mf{h}$ is the Lie algebra of $H$ and $\mf{m}$ is an $\Ad(H)$-invariant subspace complementary to $\mf{h} \subseteq \mf{g}$.  The representation of $H$ on $\mf{m}$ is called the isotropy representation of $H$.  Homogeneous metrics on $G/H$ are in one-to-one correspondence with $\Ad(H)$-invariant inner products on $\mf{m}$.  Furthermore, if we let $\met_{\mf{m}}$ be an $\Ad(H)$-invariant inner product on $\mf{m}$ and $\met_{\mf{h}}$ an arbitrary inner product on $\mf{h}$, then we may define an inner product $\met$ on $\mf{g}$ by declaring $\mf{h} \perp \mf{m}$.  Via left translation we get a left-invariant metric (also denoted by $\met$) on $G$ and a homogeneous metric (also denoted by $\met_{\mf{m}}$) on $G/H$ for which the map $\pi: (G, \met) \lra (G/H, \met_{\mf{m}})$ is a Riemannian submersion.

So, in order to understand homogeneous metrics on $G/H$, we need to understand $\Ad(H)$-invariant inner products on $\mf{m}$.  Now suppose $\mf{m}$ splits as $\mf{m} = \mf{p}_1 \oplus \cdots \oplus \mf{p}_s$ into a sum of $\Ad(H)$ irreducible sub-modules; the following well-known lemma follows readily from Schur's Lemma.

\begin{lem}
\label{perpreps}
Let $\mf{g} = \mf{h} \oplus \mf{m}$ be as above, where $\mf{m} = \mf{p}_1 \oplus \cdots \oplus \mf{p}_s$ and $\mf{p}_k$ is $\Ad(H)$ irreducible for all $1 \leq k \leq s$, and let $\met_{\mf{m}}$ be an $\Ad(H)$-invariant inner product on $\mf{m}$. Then $\mf{p}_i \perp \mf{p}_j$ with respect to $\met_{\mf{m}}$ whenever $\mf{p}_i$ and $\mf{p}_j$ are inequivalent representations of $H$.
\end{lem}

As it turns out, $\Ad(H)$-invariant inner products on the irreducible summands $\mf{p}_k$ are very special. The following lemma is also well known.

\begin{lem}
\label{uniquemetrics}
Let $H$ be any group and let $V$ be an irreducible $H$-representation. Suppose there are two $H$-invariant inner products, $\met_1$ and $\met_2$, on $V$. Then there exists a constant $\lambda >0$ such that $\met_1 = \lambda\, \met_2$.
\end{lem}

In the special case where the irreducible summands $\mf{p}_1, \dots, \mf{p}_s$ of $\mf{m}$ are pairwise inequivalent, Lemmas \ref{perpreps} and \ref{uniquemetrics} tell us that \emph{all} homogeneous metrics on $G/H$ are described by $s$ positive real numbers, namely
\begin{equation}
\label{inequivmet}
\met_{\mf{m}} = \lambda_1 Q|_{\mf{p}_1} \perp \lambda_2 Q|_{\mf{p}_2} \perp \dots \perp \lambda_s Q|_{\mf{p}_s},
\end{equation}
where $Q$ is some bi-invariant metric on $G$ and $\lambda_1, \dots, \lambda_s >0$.  We will always choose $Q$ to be the negative of the Killing form on $\mf{g}$.

Suppose now that some of our irreducible $\Ad(H)$ sub-modules are pairwise equivalent.  In this situation it is more complicated to write down all possible homogeneous metrics on $G/H$ because equivalent sub-modules need not be perpendicular.  However, there is a well-established procedure.  Any $\Ad(H)$-invariant inner product $\met_{\mf{m}}$ on $\mf{m}$ satisfies $\MMet{X}{Y}_{\mf{m}} = Q(\Phi(X), Y)$, where $\Phi: \mf{m} \lra \mf{m}$ is a linear, positive definite, symmetric, $\Ad(H)$-equivariant map.  Therefore, the space of all possible $\Ad(H)$-invariant inner products on $\mf{m}$ may be described by parametrizing the space of all possible maps $\Phi$.  This is done as follows.

We first consider the complexification $\psi \otimes \C$ of a real representation $\psi : G \lra \Aut(V)$.  If $\psi \otimes \C$ is irreducible, we say $\psi$ is orthogonal.  Otherwise $\psi \otimes \C = \varphi \oplus \bar\varphi$.  If $\varphi$ is not equivalent to $\bar\varphi$, we say $\psi$ is unitary.  If, on the other hand, $\varphi$ and $\bar\varphi$ are equivalent, we say $\psi$ is symplectic.  We call a map $A: V \lra V$ such that $\psi \comp A = A \comp \psi$ an \emph{intertwining operator}.  The space of all intertwining operators is has dimension one if $\psi$ is orthogonal, two if $\psi$ is unitary, and four if $\psi$ is symplectic.

It follows that between each pair of equivalent irreducible representations $\mf{p}_i, \mf{p}_j$ we have either a one, two, or four parameter family of $\Ad(H)$-invariant inner products.  That is, $\MMet{\mf{p}_i}{\mf{p}_j}_{\mf{m}} = Q(\Phi(\mf{p}_i), \mf{p}_j)$ is given by one, two or four real parameters.  Therefore $\Phi$ may be represented by an $s \times s$ symmetric matrix whose $ij$-th entry is real when $i=j$, zero if $\mf{p}_i$ and $\mf{p}_j$ are inequivalent, and an element of $\R$, $\C$ or $\H$ when $\mf{p}_i$ and $\mf{p}_j$ are equivalent.

\medskip
Let us return now to the second part of the construction suggested in Section \ref{basic}.  Consider the situation where we have $G/K_1 = K_2/H$ as homogeneous spaces and a chain of subgroups $H \subseteq L \subseteq K_2$ which gives the homogeneous fibration, $L/H \lra K_2/H \stackrel{\pi_2}{\lra} K_2/L$. We fix a bi-invariant metric $\met_0$ on $G$ and hence a normal homogeneous metric on $G/K_1$.  It is clear that $K_2$ acts isometrically and transitively on $G/K_1$ with isotropy group $H$.  Therefore there is some homogeneous metric on $K_2/H$ isometric to the normal homogeneous metric on $G/K_1$.  We want to choose this metric on $K_2/H$ and then determine whether the map $\pi_2 : K_2/H \lra K_2/L$ is a Riemannian submersion.

Consider the Lie algebras $\mf{h} \subseteq \mf{l} \subseteq \mf{k}_2$ corresponding to the Lie groups $H \subseteq L \subseteq K_2$.  If we choose an $\Ad(H)$-invariant complement $\mf{m}_1$ of $\mf{h} \subseteq \mf{l}$ and an $\Ad(L)$-invariant complement $\mf{m}_2$ of $\mf{l} \subseteq \mf{k}_2$, then we arrive at a decomposition
$$
\mf{k}_2 = \mf{l} \oplus \mf{m}_2 = (\mf{h} \oplus \mf{m}_1) \oplus \mf{m}_2.
$$
In particular, $\mf{m}_1 \oplus \mf{m}_2$ is an $\Ad(H)$-invariant complement of $\mf{h} \subseteq \mf{k}_2$  since the $H$ action on $\mf{m}_2$ is simply a restriction of the $L$ action.  We remark that $\mf{m}_1$ and $\mf{m}_2$ correspond to the tangent spaces of the fiber and base of the fibration $L/H \lra K_2/H \stackrel{\pi_2}{\lra} K_2/L$ respectively.

Let $\mf{m}_2 = \mf{q}_1 \oplus \cdots \oplus \mf{q}_s$ be the irreducible decomposition of $\mf{m}_2$ with respect to $\Ad(L)$.  From our discussion above we can therefore determine all homogeneous metrics on $K_2/L$.  Recall that we require $\mf{m}_2 \perp \mf{l}$.  In particular, we see that a necessary condition for $\pi_2$ to be a Riemannian submersion is $\mf{m}_1 \perp \mf{m}_2$ with respect to the homogeneous metric on $K_2/H$.

Consider now homogeneous metrics on $K_2/H$.  Let $\mf{m}_1 = \mf{p}_1 \oplus \cdots \oplus \mf{p}_r$ be the irreducible decomposition of $\mf{m}_1$ with respect to $\Ad(H)$.  In general, each of the $\Ad(L)$ irreducible summands $\mf{q}_j \subseteq \mf{m}_2$, $1 \leq j \leq s$, will split further into $\Ad(H)$ irreducible summands.  This is usually a problem when we want $\pi_2$ to be a Riemannian submersion (given by restriction of the inner product on $\mf{m}_1 \oplus \mf{m}_2$ to $\mf{m}_2$).  Together with the discussion in the previous paragraph, this leads us to consider a special case.  Suppose that the following conditions hold:

\smallskip
\noindent (\textit{i}) $\mf{q}_1 \oplus \cdots \oplus \mf{q}_s$ is the irreducible decomposition of $\mf{m}_2$ with respect to both $\Ad(H)$ and $\Ad(L)$;

\medskip
\noindent (\textit{ii}) For all $1 \leq i \leq r$, $1 \leq j \leq s$, the $\Ad(H)$ irreducible representations $\mf{p}_i$ and $\mf{q}_j$ are pairwise inequivalent;

\medskip
\noindent (\textit{iii}) If $\mf{q}_i$ and $\mf{q}_j$,  $i,j \in \{1, \dots, s \}$, are two equivalent irreducible representations, then they are of the same type with respect to both $\Ad(H)$ and $\Ad(L)$, i.e. $\mf{q}_i$ and $\mf{q}_j$ are either both orthogonal, both unitary or both symplectic as both $H$ and $L$ representations.

\smallskip

Conditions (\textit{i}) and (\textit{ii}) ensure, by Lemma \ref{perpreps}, that $\mf{m}_1 \perp \mf{m}_2$ for every homogeneous metric on $K_2/H$.  Conditions (\textit{i}) and (\textit{iii}) (together with Lemmas \ref{perpreps} and \ref{uniquemetrics}) ensure that the restriction of a homogeneous metric on $K_2/H$ to $\mf{m}_2$ yields a homogeneous metric on $K_2/L$.  Therefore $\pi_2$ gives a Riemannian submersion for \emph{any} choice of homogeneous metric on $K_2/H$.  In particular, when $K_2/H$ is isometric to the normal homogeneous space $G/K_1$, we obtain a Riemannian submersion $\pi: (G, \met_0) \lra K_2/L$ as desired.  We have proved:

\begin{thm}
Suppose we have $G/K_1 = K_2/H$, where $G$ is a compact, semi-simple Lie group with bi-invariant metric $\met_0$ and $K_1, K_2, H$ are closed subgroups of $G$. If, for some closed subgroup $H\subseteq L \subseteq K_2$, conditions (\textit{i}),  (\textit{ii}) and (\textit{iii}) above hold, then there is a Riemannian submersion from $(G, \met_0)$ onto $K_2/L$.
\end{thm}

We are now ready to discuss the candidates from Table \ref{table1}.

\section{$\SO(2n)  \lra \sph^{2n-2}$, $n \geq 4$}
\label{evensphere1}

\begin{thm}
\label{evensphmet}
For each $n \geq 2$, there is a Riemannian submersion
$$
(\SO(2n), \met_0) \lra \sph^{2n-2}.
$$
\end{thm}

\begin{proof}
Consider the Riemannian submersion $\pi_1 : \SO(2n) \lra \SO(2n)/\U(n)$, where we have equipped $\SO(2n)$ with a bi-invariant metric $\met_0$.  From Oni\v s\v cik's classification we know that $\SO(2n)/\U(n) = \SO(2n-1)/\U(n-1)$.

Now $\U(n-1) \subseteq \SO(2n-2) \subseteq \SO(2n-1)$ and so we have a fibration
$$
\SO(2n-2)/\U(n-1) \lra \SO(2n-1)/\U(n-1) \stackrel{\pi_2}{\lra} \SO(2n-1)/\SO(2n-2) = \sph^{2n-2}
$$
The tangent space to the base may be identified with $\mf{p}_2$, a $2(n-1)$-dimensional, $\Ad(\SO(2n-2))$-irreducible complement of $\mf{so}(2n-2) \subseteq \mf{so}(2n-1)$.  The restriction of the $\Ad(\SO(2n-2))$ action to $\U(n-1) \subseteq \SO(2n-2)$ is the standard irreducible representation of $\U(n-1)$ on $\mf{p}_2 \cong \C^{n-1}$.

On the other hand, the tangent space to the fiber may be indentified with $\mf{p}_1$, an $\Ad(\U(n-1))$-invariant complement of $\mf{u}(n-1) \subseteq \mf{so}(2n-2)$.  $\mf{p}_1$ is $(n-1)(n-2)$-dimensional and is $\Ad(\U(n-1))$-irreducible (see for instance \cite{Ke}).

Thus we may write
$$
\begin{aligned}
\mf{so}(2n-1) &= \mf{so}(2n-2) \oplus \mf{p}_2 \\
              &= (\mf{u}(n-1) \oplus \mf{p}_1) \oplus \mf{p}_2
\end{aligned}
$$
where $\mf{p}_1$ and $\mf{p}_2$ are orthogonal by the inequivalence of the $\U(n-1)$ representations.  For $n \neq 4$ this is clear for dimension reasons, while the case $n = 4$ follows from the discussion in \cite{Ke}.

Hence all homogeneous metrics on $\SO(2n-1)/\U(n-1)$ are given by
$$
\met = \lambda_1 Q|_{\mf{p}_1} \perp \lambda_2 Q|_{\mf{p}_2},
$$
where $Q(X,Y) = -\frac{1}{2}\tr(XY)$ (in particular, $\Ad(\U(n-1))$-invariant) and $\lambda_1, \lambda_2 >0$.  We choose $\lambda_1$ and $\lambda_2$ such that $\SO(2n-1)/\U(n-1)$ is isometric to $\SO(2n)/\U(n)$ equipped with the normal homogeneous metric (from \cite{Ke} it follows that the appropriate choice is $\lambda_2 = \frac{1}{2}\lambda_1$).  Furthermore, since $\mf{p}_2$ is $\Ad(\SO(2n-2))$-irreducible, perpendicular to $\mf{so}(2n-2)$ and equipped with an $\Ad(\SO(2n-2))$-invariant metric, the map
$$
\pi_2: \SO(2n-1)/\U(n-1) \lra \SO(2n-1)/\SO(2n-2) = \sph^{2n-2}
$$
is a Riemannian submersion.

The composition $\pi = \pi_2 \comp \pi_1$ is the desired Riemannian submersion from $\SO(2n)$ (equipped with a bi-invariant metric) to $\sph^{2n-2}$.
\end{proof}

Note that when $n = 2$ we have $\Delta \SO(2) \backslash \SO(4) / \SO(3) = \sph^2$ and when $n = 3$ we have $\SO(3) \backslash \SO(6) / \SU(3) = \Delta \SU(2) \backslash \SU(4) / \SU(3) = \hp^1 = \sph^4$, where $\Delta$ denotes the diagonal embedding in both cases.  On the other hand, for $n \geq 4$:

\begin{thm}
\label{evensphtop}
For each $n\geq 4$, there is no Lie group $U$ acting freely on $\SO(2n)$ such that $\SO(2n)/ U = \sph^{2n-2}$.
\end{thm}

\begin{proof}
Suppose there is some Lie group $U$ acting freely on $\SO(2n)$, $n \geq 4$, such that $\sph^{2n-2} = \SO(2n)/ U$.  Then we have a fibration $U \lra \SO(2n) \lra \sph^{2n-2}$.  The long exact sequence of homotopy groups for this fibration yields $\pi_1(U) = \Z_2$ and $\pi_3(U) = \Z$.  Therefore, $U$ must be a simple Lie group of dimension $(2n-1)(n-1) +1$.

Consider first the case $n>4$. Then from the long exact sequence in homotopy and the stable homotopy groups of Lie groups we see that
$$
\cdots \underset{\underset{0}{\|}}{\pi_6(\sph^{2n-2})} \rightarrow \pi_5(U) \rightarrow \underset{\underset{0}{\|}}{\pi_5(\SO(2n))} \rightarrow \pi_5(\sph^{2n-2}) \rightarrow \cdots
$$
which forces $\pi_5(U) = 0$. Since $\dim(U) = (2n-1)(n-1)+1 \geq 37$, we are in the stable range and may therefore conclude that either $U \cong \SO(m)$ or $U$ is an exceptional simple group. A quick check reveals that $\dim(U)$ is never equal to the dimension of any exceptional group. On the other hand, we see evidently that $\dim(\SO(2n-1)) = (2n-1)(n-1) < \dim(U) < \dim(\SO(2n)) = (2n-1)n$.

When $n=4$, the dimension of $U$ is 22 and there is no simple Lie group of that dimension.
Hence there are no Lie groups $U$ for which $\SO(2n)/ U = \sph^{2n-2}$ for each $n \geq 4$.
\end{proof}


\section{$\SU(2n)  \lra \sph^{4n-3}$ and $\SU(2n)\lra \cp^{2n-2}$, $n \geq 3$}

\begin{thm}
\label{oddsphmet}
For each $n \geq 3$, there are Riemannian submersions
$$
(\SU(2n), \met_0) \lra \sph^{4n-3},\qquad (\SU(2n), \met_0) \lra \cp^{2n-2}
$$
Moreover, there are no groups $U,U' \subseteq {\rm Diff}(\SU(2n))$ so that $\SU(2n)/U = \sph^{4n-3}$ and $\SU(2n)/U' = \cp^{2n-2}$.
\end{thm}

The arguments in this case are essentially identical to the case of $\SO(2n) \lra \sph^{2n-2}$ so we omit them. The only comment that may be of some independent interest is the choice of constants for the homogeneous metric on $\SU(2n-1)/\syp(n-1)$ to be isometric to the normal homogeneous metric on $\SU(2n)/\syp(n)$. The isotropy representation of $\syp(n-1)\subseteq \SU(2n-1)$ splits into three irreducible summands, $\mathfrak{su}(2n-1) = \mathfrak{sp}(n-1) \oplus \mathfrak{p}_1 \oplus \mathfrak{p}_2 \oplus \mathfrak{p}_3$, where $\mathfrak{sp}(n-1) \oplus \mathfrak{p}_1 = \mathfrak{su}(2n-2) = \mathfrak{l}$, $\dim(\mf{p}_2) = 1$ and $\dim(\mf{p}_3) = 4(n-1)$. All homogeneous metrics on $\SU(2n-1)/\syp(n-1)$ are given by
$$
\langle\, ,\, \rangle = \lambda_1 Q\mid_{\mathfrak{p}_1} \perp \lambda_2 Q\mid_{\mathfrak{p}_2} \perp \lambda_3 Q\mid_{\mathfrak{p}_3},
$$
where $Q(X,Y) = -\frac{1}{2} {\rm tr}(XY)$ is a bi-invariant metric. To be isometric to the normal homogeneous space $\SU(2n)/\syp(n)$, it follows from \cite{Ke} that the appropriate choices are: $\lambda_2 = \frac{n}{2n-1}\lambda_1, \lambda_3 = \frac{1}{2} \lambda_1$.

\section{$\SO(4n)  \lra V_{3}(\R^{4n-1})$, $n \geq 3$}

\begin{thm}
\label{stiefelmet}
For each $n \geq 3$, there is a Riemannian submersion
$$
(\SO(4n), \met_0) \lra V_{3}(\R^{4n-1}),
$$
where $V_{3}(\R^{4n-1})$ is the Stiefel manifold $\SO(4n-1)/ \SO(4n-4)$.
\end{thm}

\begin{proof}
Consider the Riemannian submersion $\pi_1 : \SO(4n) \lra \SO(4n)/\syp(n)\syp(1)$, where $\SO(4n)$ is equipped with a bi-invariant metric.  From Oni\v s\v cik's classification we know that $\SO(4n)/\syp(n)\syp(1) = \SO(4n-1)/\syp(n-1)\syp(1)$.
Now $\syp(n-1)\syp(1) \subseteq \SO(4n-4) \subseteq \SO(4n-1)$ and so we have a fibration
$$
\xymatrix{
\SO(4n-4)/\syp(n-1)\syp(1) \ar[r] & \SO(4n-1)/\syp(n-1)\syp(1) \ar[d]^{\pi_2} & \\
 &  \SO(4n-1)/\SO(4n-4) & \hspace{-1.3cm} = V_{3}(\R^{4n-1})
}
$$
The tangent space to the fiber may be indentified with $\mf{p}_1$, an $\Ad(\syp(n-1)\syp(1))$-invariant complement of $\mf{so}(4n-4) \subseteq \mf{so}(4n-1)$.  $\mf{p}_1$ is $3(2n-1)(n-2)$-dimensional and in \cite[1984]{Wo} it is shown that it is $\Ad(\syp(n-1)\syp(1))$-irreducible.

On the other hand, we may use the chain of subgroups
$$
\SO(4n-4) \subseteq \SO(4n-3) \subseteq \SO(4n-2) \subseteq \SO(4n-1)
$$
to identify the tangent space to the base with
\begin{equation}
\label{mdecomp}
\mf{m} := \mf{p}_2 \oplus \mf{p}_3 \oplus \mf{p}_4 \oplus \mf{p}_5 \oplus \mf{p}_6 \oplus \mf{p}_7 \subseteq \mf{so}(4n-1)
\end{equation}
where $\mf{p}_i \cong \R^{4n-4}$, $i = 2,3,4$, and $\mf{p}_i \cong \R$, $i = 5,6,7$, are $\Ad(\SO(4n-4))$-irreducible.  The isotropy representation of $\SO(4n-4)$ on $\mf{m}$ decomposes into standard $\SO(4n-4)$ actions on $\mf{p}_i \cong \R^{4n-4}$, $i = 2,3,4$, and trivial representations on $\mf{p}_i \cong \R$, $i = 5,6,7$.  This is easily seen by considering the $\Ad(\SO(4n-4))$ action on
\begin{equation}
\label{matdecomp}
\mf{so}(4n-1) =
\left(\begin{array}{ccc|c|c|c}
    & & & \vdots & \vdots & \vdots \\
    & \mf{so}(4n-4) & & \mf{p}_2 & \mf{p}_3 & \mf{p}_4 \\
    & & & \vdots & \vdots & \vdots \\
\hline
    & \cdots & & 0 & \mf{p}_5 & \mf{p}_6 \\
\hline
    & \cdots & & \cdot & 0 & \mf{p}_7 \\
\hline
    & \cdots & & \cdot & \cdot & 0
\end{array}\right)
\end{equation}
where we recall that elements of $\mf{so}(k)$ are skew-symmetric.  It is clear that the representations $\mf{p}_2$, $\mf{p}_3$ and $\mf{p}_4$ are equivalent \emph{real} representations, as are $\mf{p}_5$, $\mf{p}_6$ and $\mf{p}_7$.  Moreover, by Schur's Lemma, $(\mf{p}_2 \oplus \mf{p}_3 \oplus \mf{p}_4) \perp (\mf{p}_5 \oplus \mf{p}_6 \oplus \mf{p}_7)$.  In each case $\mf{p}_i \otimes \C$ is irreducible.  Hence the space of intertwining operators is one-dimensional when $\mf{p}_i$ and $\mf{p}_j$ are equivalent, from which it follows that the space of all $\Ad(\SO(4n-4))$-invariant inner products on $\mf{m}$ is given by two real, symmetric, $3 \x 3$ matrices, i.e., $12$ real parameters.

We now consider the isotropy representation of $\syp(n-1) \syp(1)$ and check whether the type and irreducible decomposition of the representation restricted from $\SO(4n-4)$ remains the same. We remark that this is crucial otherwise the number of parameters that determine the metric may be different and hence, likely, not yield a Riemannian submersion.
The restriction of the $\Ad(\SO(4n-4))$ action on $\mf{m}$ to $\syp(n-1)\syp(1) \subseteq \SO(4n-4)$ yields the same irreducible decomposition as in (\ref{mdecomp}).  An easy way to see this is by considering the subgroup $\syp(n-1) \subseteq \syp(n-1)\syp(1)$.  It's clear that this gives the same decomposition as in (\ref{mdecomp}), where the $\Ad(\syp(n-1))$ action on $\mf{p}_i$, $i = 2,3,4$, is the standard irreducible representation of $\syp(n-1)$ on $\R^{4n-4} \cong \H^{n-1}$.  Thus $\syp(n-1)\syp(1)$ must also decompose $\mf{m}$ as in (\ref{mdecomp}).

In \cite[1984]{Wo} it is shown that the embedding of $\syp(n-1)\syp(1)$ into $\SO(4n-4)$, namely the restriction of the standard (complex) $\SO(4n-4)$ representation to $\syp(n-1)\syp(1)$, is given by the tensor product of the standard $\syp(n-1)$ and $\syp(1)$ (complex) representations.  Since each of these is a sympletic representation, it follows from \cite[p. 264, Exer. 3]{BtD} that their tensor product is an orthogonal representation, i.e. $\mf{p}_i \otimes \C$ is $\syp(n-1)\syp(1)$-irreducible for $i=2, 3, 4$.  A similar argument works for $\mf{p}_i \otimes \C \cong \C$, $i = 5,6,7$.  Hence the space of intertwining operators is one-dimensional whenever $\mf{p}_i$ and $\mf{p}_j$ are equivalent.

Thus we may write
$$
\begin{aligned}
\mf{so}(4n-1) &= \mf{so}(4n-4) \oplus \mf{m} \\
              &= (\mf{sp}(n-1)\mf{sp}(1) \oplus \mf{p}_1) \oplus \mf{m}.
\end{aligned}
$$
Since $\dim (\mf{p}_1) \neq \dim (\mf{p}_i)$ for all $i = 2, \dots, 7$, Schur's Lemma ensures that $\mf{p}_1 \perp \mf{m}$ for every $\Ad(\syp(n-1)\syp(1))$-invariant inner product on $\mf{p}_1 \oplus \mf{m}$.  Therefore it follows that the space of all $\Ad(\syp(n-1)\syp(1))$-invariant inner products on $\mf{p}_1 \oplus \mf{m}$ is given by one real parameter together with two real, symmetric, $3 \x 3$ matrices, i.e., $13$ real parameters.

In particular, for \emph{any} homogeneous metric on $\SO(4n-1)/\syp(n-1)\syp(1)$, the map $\pi_2: \SO(4n-1)/\syp(n-1)\syp(1) \lra \SO(4n-1)/\SO(4n-4)$ is a Riemannian submersion, where the metric on $\SO(4n-1)/\SO(4n-4)$ is given by restricting the $\Ad(\syp(n-1)\syp(1))$-invariant inner product on $\mf{p}_1 \oplus \mf{m}$ to $\mf{m}$.

Hence, if we choose the $13$ real parameters describing the homogeneous metric such that the metric on $\SO(4n-1)/\syp(n-1)\syp(1)$ is isometric to the normal homogeneous metric on $\SO(4n)/\syp(n)\syp(1)$, then the composition
$$
\pi = \pi_2 \comp \pi_1: SO(4n) \lra \SO(4n-1)/\SO(4n-4) = V_{3}(\R^{4n-1})
$$
is a Riemannian submersion as desired.
\end{proof}

\begin{thm}
\label{stiefeltop}
For each $n\geq 3$, there is no Lie group $U$ acting freely on $\SO(4n)$ such that $\SO(4n)/ U = V_3(\R^{4n-1})$.
\end{thm}

\begin{proof}
Suppose there is some Lie group $U$ acting freely on $\SO(4n)$, $n \geq 3$, such that $V_{3}(\R^{4n-1}) = \SO(4n)/ U$.  Then we have a fibration $U \lra \SO(4n) \lra V_{3}(\R^{4n-1})$.  It is well-known that $V_{k}(\R^{m})$ is $(m-k-1)$-connected \cite[p. 382]{Ha}.  In particular, $\pi_j(V_{3}(\R^{4n-1})) = 0$ for all $j \leq 7$ since $n \geq 3$.  The long exact sequence of homotopy groups for our fibration now yields $\pi_1(U) = \Z_2$ and $\pi_3(U) = \Z$.  Therefore, $U$ must be a simple Lie group.

Since $V_3(\R^{4n-1})$ is at least 7-connected, we see from the long exact sequence in homotopy that $\pi_5(U) = \pi_5(\SO(4n)) = 0$. Since $\dim(U) = 8n^2 - 14n +9 \geq 39$, we are in the stable range and it follows that $U$ must be isomorphic to $\SO(m)$ for some $m$ or to an exceptional simple group. $\dim(U)$ is not equal to that of any exceptional group. On the other hand,
$$
\dim(\SO(4n-3)) = 8n^2 - 14n+6 < \underset{= \dim(U)}{8n^2 -14n +9} < 8n^2 - 10n+3 = \dim(\SO(4n-2))
$$

Hence there are no Lie groups $U$ for which $\SO(4n)/ U = V_3(\R^{4n-1})$ if $n \geq 3$.
\end{proof}

\begin{cor}
\label{stiefelcor}
For each $n \geq 3$, there is a Riemannian submersion
$$
(\SO(4n), \met_0) \lra M^{12n-10} := \SO(2) \backslash \SO(4n-1) /\SO(4n-4).
$$
Moreover, this Riemannian submersion is not the result of a free, isometric Lie group action on $\SO(4n)$.
\end{cor}

\begin{proof}
Consider the circle subgroup $\SO(2) \subseteq \SO(4n-1)$ given by $\diag(A,\dots,A,1)$, $A \in \SO(2)$.  Then $\SO(2)$ acts freely on $V_3 (\R^{4n-1}) = \SO(4n-1) /\SO(4n-4)$ on the left since the two-sided action of $\SO(2) \x \SO(4n-4)$ on $\SO(4n-1)$ is free.  Now, since the metric on $V_3 (\R^{4n-1})$ described in Theorem \ref{stiefelmet} is homogeneous, this $\SO(2)$-action is by isometries.  Therefore $V_3 (\R^{4n-1}) \lra M^{12n-10}$ is a Riemannian submersion and we may compose it with $(\SO(4n), \met_0) \lra V_3 (\R^{4n-1})$ to yield the desired Riemannian submersion.

Consider the long exact sequence for homotopy associated to the fibration
$$
\sph^1 \lra V_3(\R^{4n-1}) \lra M.
$$
Since $\pi_j(V_{3}(\R^{4n-1})) = 0$ for all $j \leq 7$, it follows that $\pi_2(M) = \Z$ and $\pi_j(M) = 0$ for $ j = 1, 3, 4, 5, 6$.
Suppose that there is some Lie group $U'$ acting freely on $\SO(4n)$ such that $M = \SO(4n)/U'$.  The long exact homotopy sequence for the fibration $U' \lra \SO(4n) \lra M$ shows that $U'$ is diffeomorphic to either $\sph^1 \x U$ or $(\sph^1 \x U)/\Z_2$, where $U$ is a compact, connected, simply connected, simple Lie group.

From the long exact sequence it also follows that $\pi_5(U') = \pi_5(U) = \pi_5(\SO(4n)) = 0$. So we are now looking for $U$, a compact, simple group of dimension $8n^2 - 14n+9$ and isomorphic to $\SO(m)$ for some $m$ or to an exceptional simple group. From the proof of Theorem~\ref{stiefeltop} there is no such $U$ and hence, there can be no free $U'$-action on $\SO(4n)$ with quotient $M$.
\end{proof}

\section{$\SO(4n) \rightarrow \sph^{4n-2}$ and $\SO(4n) \rightarrow T^1 \sph^{4n-2}$}
\label{evensphere2}

These two examples yield topological submersions which are not group quotients. Indeed, one may adapt the proof of Theorem~\ref{stiefeltop} to show that neither $\sph^{4n-2}$ nor $T^1 \sph^{4n-2}$ are quotients of $\SO(4n)$ by a group action. However, our method of constructing a Riemannian submersion breaks down in this instance.

Consider our setup: we start with a bi-invariant metric on $\SO(4n)$ which yields a homogeneous metric $\met$ on $\SO(4n-1)/\syp(n-1)\syp(1)$ isometric to the normal homogeneous metric on $\SO(4n)/\syp(n)\syp(1)$.
From the previous section we already know the isotropy representation of $\syp(n-1)\syp(1)$:
$$
\mathfrak{so}(4n-1) = \mathfrak{sp}(n-1) \mathfrak{sp}(1) \oplus \mathfrak{p}_1 \oplus \underbrace{\mathfrak{p}_2 \oplus \mathfrak{p}_3 \oplus \mathfrak{p}_4 \oplus \mathfrak{p}_5 \oplus \mathfrak{p}_6 \oplus \mathfrak{p}_7}_{\mathfrak{m}}
$$
where $\mathfrak{p}_1$ is the complement of $\mathfrak{sp}(n-1) \mathfrak{sp}(1)$ in $\mathfrak{so}(4n-4)$ and $\mathfrak{m}$ decomposes into six irreducible pieces: three equivalent modules isomorphic to $\R^{4n-4}$ and three trivial (one dimensional) modules. The decomposition of $\mathfrak{m}$ is the same for $\SO(4n-4)$ as it is for $\syp(n-1) \syp(1)$.
Recall that in the Lie algebra $\mathfrak{so}(4n-1)$ this decomposition is given by $(\ref{matdecomp})$.

Consider now the quotient $\SO(4n-1)/\SO(4n-3) = T^1 \sph^{4n-2}$. The isotropy representation splits as $\mathfrak{so}(4n-1) = \mathfrak{so}(4n-3) \oplus \mathfrak{n}$ which decomposes as:
$$
\mathfrak{so}(4n-1) = \underbrace{(\mathfrak{sp}(n-1)\mathfrak{sp}(1) \oplus \mathfrak{p}_1 \oplus \mathfrak{p}_2)}_{\mathfrak{so}(4n-3)} \oplus \underbrace{(\mathfrak{q}_3 \oplus \mathfrak{q}_4 \oplus \mathfrak{p}_7)}_{\mathfrak{n}}
$$
where $\mathfrak{q}_3, \mathfrak{q}_4$ are equivalent, irreducible modules isomorphic to $\R^{4n-3}$ and $\mathfrak{p}_7$ is a trivial, one dimensional module. Note that $\mathfrak{q}_j$ splits further under the action of $\syp(n-1)\syp(1)$ as $\mathfrak{q}_3 = \mathfrak{p}_3 \oplus \mathfrak{p}_5$ and $\mathfrak{q}_4 = \mathfrak{p}_4 \oplus \mathfrak{p}_6$. Therefore, the two isotropy actions have different irreducible decompositions.

In order for the maps, $\pi_2: \SO(4n-1)/\syp(n-1)\syp(1) \rightarrow \SO(4n-1)/\SO(4n-3)$ and $\pi_2': \SO(4n-1)/\syp(n-1)\syp(1) \rightarrow \SO(4n-1)/\SO(4n-2)$ to be Riemannian submersions we need (see Section 3) that $\mathfrak{p}_2$ is perpendicular to $\mathfrak{p}_4$ with respect to $\met$. 
So we need to know the induced left invariant metric on $\mathfrak{p}_2 \oplus \cdots \oplus \mathfrak{p}_7$ inside $\mathfrak{so}(4n-1)$. This is done as follows: restrict the bi-invariant metric on $\mathfrak{so}(4n)$ which is given by $\langle X,Y \rangle_0 = -\frac{1}{2} {\rm tr}(XY)$, to $\mathfrak{so}(4n-1)$. The tangent space to $\SO(4n-1)/\syp(n-1)\syp(1)$ is isomorphic to $\mathfrak{p}_1 \oplus \cdots \oplus \mathfrak{p}_7$. On the other hand, we also have $\mathfrak{so}(4n) = \mathfrak{sp}(n) \oplus \mathfrak{sp}(1) \oplus \mf{r}$; let $\pi_{\mf{r}}: \mathfrak{so}(4n) \rightarrow \mf{r}$ denote the orthogonal projection. If $U,V$ are vectors in $\mathfrak{p}_1 \oplus \cdots \oplus \mathfrak{p}_7$, then the induced metric is given by, $\MMet{U}{V} = \MMet{\pi_{\mf{r}}(U)}{\pi_{\mf{r}}(V)}_0$.

Let $E_{ij} \in \mathfrak{so}(4n)$ denote the vector whose $ij$-th entry is 1 (and therefore its $ji$-th entry is necessarily $-1$). Then the $E_{ij}$ form an orthogonal basis for $\mathfrak{so}(4n)$. Consider now the vectors, $E_{1, 4n-3} \in \mathfrak{p}_2$, 
$E_{3, 4n-1} \in \mathfrak{p}_4$.  A simple calculation reveals that
$\MMet{E_{1,4n-3}}{E_{3,4n-1}} = \MMet{\pi_{\mf{r}}(E_{1,4n-3})}{\pi_{\mf{r}}(E_{3,4n-1})}_0 = -\frac{1}{4}$. This shows immediately that the subspaces are pairwise not orthogonal, as claimed, and hence the maps $\pi_2$ and $\pi_2'$
are \textbf{not} Riemannian submersions for the metric $\met$ on $\SO(4n-1)/\syp(n-1)\syp(1)$.

\begin{remark}
One can, in fact, show that there is no homogeneous metric on $\SO(4n-1)/\syp(n-1)\syp(1)$ whatsoever such that the maps $\pi_2$ (resp. $\pi_2'$) and $\pi = \pi_2 \circ \pi_1$ (resp. $\pi' = \pi_2' \circ \pi_1$) are both Riemannian submersions.
\end{remark}



\appendix

\section{Enlargements of transitive actions}

Table 4 is due to A.\ L.\ Oni\v s\v cik (\cite{On}) and classifies simple, compact Lie algebras $\mf{g}$ with sub-algebras $\mf{k}_1, \mf{k}_2$ such that $\mf{g} = \mf{k}_1 + \mf{k}_2$. We present the group versions here and identify the space whenever possible.

\begin{center}
\large
\begin{table}[!htp]
\begin{tabular}{|l|l|l|}\hline
\multicolumn{1}{|c|}{\multirow{2}{*}{$\boldsymbol{G/K_1}$}} & \multicolumn{1}{|c|}{\multirow{2}{*}{$\boldsymbol{K_2/H}$}} & \textbf{Homogeneous} \\
 & & \multicolumn{1}{|c|}{\textbf{space}} \\ \hline \hline
 \multicolumn{3}{|c|}{$\boldsymbol{G/K_1}$ \textbf{symmetric}} \\ \hline \hline
$\SO(4n)/\SO(4n-1)$ & $\syp(n)/\syp(n-1)$ & $\sph^{4n-1}$ \\ \hline
 $\SO(4n)/\SO(4n-1)$ & $\syp(n) \U(1)/\syp(n-1) \U(1)$ & $\sph^{4n-1}$ \\ \hline
 $\SO(4n)/\SO(4n-1)$ & $\syp(n) \syp(1)/\syp(n-1) \syp(1)$ & $\sph^{4n-1}$ \\ \hline
$\SO(2n)/\SO(2n-1)$ & $\U(n)/\U(n-1)$ & $\sph^{2n-1}$ \\ \hline
 $\SO(2n)/\SO(2n-1)$ & $\SU(n)/\SU(n-1)$ & $\sph^{2n-1}$ \\ \hline
$\SO(2n)/\U(n)$ & $\SO(2n-1)/\U(n-1)$ & \\ \hline 
$\SO(16)/\SO(15)$ & $\Spin(9)/\Spin(7)$ & $\sph^{15}$ \\ \hline
$\SO(8)/\SO(7)$ & $\Spin(7)/\G_2$ & $\sph^7$ \\ \hline
$\SO(8)/\Spin(7)$ & $\SO(7)/\G_2$ & $ \rp^7$ \\ \hline
$\SO(8)/\SO(3) \SO(5)$ & $\Spin(7)/\SO(4)$ & $G_3^+ (\R^8)$ \\ \hline
$ \SO(7)/\SO(6)$ & $\G_2/\SU(3)$ & $\sph^6$ \\ \hline
$\SO(7)/\G_2$ & $\SO(2) \SO(5)/\U(2)$ & $\rp^7$ \\ \hline
$\SO(7)/\SO(2) \SO(5)$ & $\G_2/\U(2)$ & $G_2^+(\R^7)$ \\ \hline
$\SU(2n)/\U(2n-1)$ & $\syp(n)/\syp(n-1) \U(1)$ & $\cp^{2n-2}$ \\ \hline
$\SU(2n)/\syp(n)$ & $\SU(2n-1)/\syp(n-1)$ & \\ \hline
\multicolumn{3}{|c|}{$\boldsymbol{G/K_1}$ \textbf{non-symmetric}} \\ \hline \hline
$\SO(4n)/\syp(n)$ & $\SO(4n-1)/\syp(n-1)$ & \\ \hline
$\SO(4n)/\syp(n) \U(1)$ & $\SO(4n-1)/\syp(n-1) \U(1)$ & \\ \hline
$\SO(4n)/\syp(n) \syp(1)$ & $\SO(4n-1)/\syp(n-1) \syp(1)$ & \\ \hline
$\SO(2n)/\SU(n)$ & $\SO(2n-1)/\SU(n-1)$ & \\ \hline
$\SO(16)/\Spin(9)$ & $\SO(15)/\Spin(7)$ & \\ \hline
$\SO(8)/\SO(6)$ & $\Spin(7)/\SU(3)$ & $V_2(\R^8)$ \\ \hline
$\SO(8)/\SO(5)$ & $\Spin(7)/\SU(2)$ & $V_3(\R^8)$ \\ \hline
$\SO(8)/\SO(2) \SO(5)$ & $\Spin(7)/\SO(2) \SU(2)$ & \\ \hline
$\SO(7)/\SO(5)$ & $\G_2/\SU(2)$ & $V_2(\R^7)$ \\ \hline

 \end{tabular} 
 \vspace{.2cm}
 \caption{Oni\v s\v cik's classification of $(G, K_1, K_2)$ with $\mf{g} = \mf{k}_1 + \mf{k}_2$, and $G$ simple.}\label{onishchik}
 \end{table}
 \end{center}


\normalsize






\begin{thebibliography}{999}

\bibitem[Be]{Be} A.\ Besse, \emph{Einstein Manifolds}, Springer, 1987


\bibitem[Bre]{Bre} G.\ E.\ Bredon, \emph{Topology and Geometry}, Springer-Verlag, New York, 1993.

\bibitem[BtD]{BtD} T.\ Br\"ocker and T.\ tom Dieck, \emph{Representations of compact Lie groups}, Springer-Verlag, 2003.


\bibitem[Es1]{Es1} R.\ Escobales, \emph{Riemannian submersions with totally geodesic fibers}, J. Diff. Geom. {\bf 10} (1975), 253-276.

\bibitem[Es2]{Es2} R.\ Escobales, \emph{Riemannian submersions from complex projective space}, J. Diff. Geom. {\bf 13} (1978), 93-107.

\bibitem[GG1]{GG1} D.\ Gromoll and K.\ Grove, \emph{A generalization of Berger's rigidity theorem for positively curved manifolds}, Ann. Sci. \`Ecole Norm. Sup. {\bf 11} (1987), 227-239.

\bibitem[GG2]{GG2} D.\ Gromoll and K.\ Grove, \emph{The low-dimensional metric foliations of Euclidean spheres}, J. Diff. Geom. {\bf 28} (1988), 143-156.

\bibitem[GW]{GW} D.\ Gromoll and G.\ Walschap, \emph{The metric fibrations of Euclidean space}, J. Diff. Geom. {\bf 57} (2001), 233-238.

\bibitem[Gro]{gro} K.\ Grove, \emph{Geometry of and via symmetries}, in Conformal, Riemannian and Lagrangian geometry, The 2000 Barrett Lectures, University Lecture Series (AMS), vol.\ 27.

\bibitem[Ha]{Ha} A.\ Hatcher, \emph{Algebraic Topology}, Cambridge University Press, 2002.


\bibitem[Ke]{Ke} M.\ Kerr, \emph{Some new homogeneous Einstein metrics on symmetric spaces}, Trans. Amer. Math. Soc., {\bf 348}(1) (1996), 153-171.


\bibitem[On]{On} A.\ L.\ Oni\v s\v cik, \emph{Inclusion relations among transitive compact transformation groups}, Am. Math. Soc. Transl., {\bf 50} (1966), 5-58.

\bibitem[Ra]{Ra} A.\ Ranjan, \emph{Riemannian submersions of spheres with totally geodesic fibres}, Osaka J. Math. {\bf 22} (1985), 243-260.







\bibitem[Wi]{Wi} B.\ Wilking, \emph{Index parity of closed geodesics and rigidity of Hopf fibrations}, Invent. Math. {\bf 144} (2001), 281-295.

\bibitem[Wo]{Wo} J.\ A.\ Wolf, \emph{The geometry and structure of isotropy irreducible homogeneous spaces}, Acta Math. {\bf 166} (1968), 59-148; Correction, Acta Math. {\bf 152} (1984), 141-142.


\bibitem[Zi]{Zi} W.\ Ziller, \emph{Examples of Riemannian manifolds with non-negative sectional curvature}, Metric and Comparison Geometry, Surv. Diff. Geom. 11, ed. K.\ Grove and J.\ Cheeger, International Press, 2007



\end{thebibliography}
\end{document}